\documentclass[twoside,11pt]{article}

\usepackage{jmlr2e}
\usepackage{amsmath}
\usepackage[pdftex]{color}
\usepackage{tikz}
\DeclareGraphicsExtensions{.jpg,.jpeg,.pdf,.png,.mps,.eps,.ps}

\let\hat\widehat

\usepackage{algorithm}
\usepackage{algpseudocode}
\algnewcommand\algorithmicinput{\textbf{INPUT:}}
\algnewcommand\INPUT{\item[\algorithmicinput]}
\algnewcommand\algorithmicoutput{\textbf{OUTPUT:}}
\algnewcommand\OUTPUT{\item[\algorithmicoutput]}

\usepackage{color}

\jmlrheading{1}{2013}{1-48}{4/00}{10/00}{Fabrizio Lecci, Alessandro Rinaldo and Larry Wasserman}

\ShortHeadings{Statistical Analysis of Metric Graph Reconstruction}{Lecci, Rinaldo and Wasserman}
\firstpageno{1}

\begin{document}

\title{Statistical Analysis of Metric Graph Reconstruction}

\author{\name Fabrizio Lecci \email lecci@cmu.edu \\
       \name Alessandro Rinaldo \email arinaldo@cmu.edu \\
       \name Larry Wasserman \email larry@cmu.edu \\
       \addr Department of Statistics\\
       Carnegie Mellon University\\
       Pittsburgh, PA 15213, USA
}


\maketitle

\begin{abstract}
\noindent A metric graph is a 1-dimensional stratified metric space consisting of
vertices and edges or loops glued together. Metric graphs can be naturally used to represent and model data
that take the form of noisy filamentary structures, such as 
street maps, neurons, networks of rivers and galaxies.  We
consider the statistical problem of reconstructing the topology of a
metric graph embedded in $\mathbb{R}^D$ from a random sample. 
We derive lower and upper bounds on the
minimax risk for the noiseless case and tubular noise case. The upper
bound is based on the reconstruction algorithm given in
\citet{aanjaneya2012metric}.
\end{abstract}


\section{Introduction}

We are concerned with the problem of estimating the topology of
filamentary data structure. Datasets consisting of points roughly
aligned along intersecting or branching filamentary paths embedded in
2 or higher dimensional spaces have become an increasingly common
type of data in a variety of scientific areas. For instance, road reconstruction 
based on GPS traces,
localization of earthquakes faults, galaxy reconstruction are 
all instances of a more general problem of
estimating basic topological features of an underlying filamentary
structure.  The recent paper by \citet{aanjaneya2012metric}, upon
which our work is based, contains further applications, as well as
numerous references.  To provide a more concrete example, consider
Figure \ref{plot3D}. The left hand side displays raw data portraying a
neuron from the hippocampus of a rat \citep{gulyas1999total}. The data
were obtained from NeuroMorpho.Org \citep{ascoli2007neuromorpho}. The
right hand side of the figure shows the output of the metric graph
reconstruction obtained using the algorithm analyzed in this paper,
originally proposed by \citet{aanjaneya2012metric}.  The
reconstruction, which takes the form of a graph, captures perfectly
all the topological features of the neuron, namely, the relationship
between the edges and vertices, the number of branching points and the
degree of each node.

Metric graphs provide the natural geometric framework for representing
intersecting filamentary structures. A metric graph embedded in a
$D$-dimensional Euclidean space ($D \geq 2$) is a 1-dimensional
stratified metric space. It consists of a finite number of points
(0-dimensional strata) and curves (1-dimensional strata) of finite
length, where the boundary of each curve is given by a pair (of
not-necessarily distinct) vertices (see the next section for a formal
definition of a metric graph). 

In this paper we study the problem of reconstructing the topology of
metric graphs from possibly noisy data, from a statistical point of
view. Specifically, we assume that we have a sample of points from a
distribution supported on a metric graph or in a small neighborhood
and we are interested in recovering the topology of the
corresponding metric graph.  To this end, we use the metric graph reconstruction algorithm given in
\citet{aanjaneya2012metric}. 
Furthermore, in our
theoretical analysis we characterize explicitly the minimal sample
size required for perfect topological reconstruction as a direct
function of parameters defining the shape of the metric graph,
introduced in Section \ref{section::background}. This leads to an upper
bound on the risk of topological reconstruction.
Finally, we obtain a lower bound on the risk of
topological reconstruction, which, in the noiseless case,
almost matches the derived upper bound, indicating that the algorithm
of \cite{aanjaneya2012metric} behaves nearly optimally.

\vspace{0.35cm}

{\em Outline.}
In Section
\ref{section::background}
we formally define metric graphs,
the statistical models we will consider and the assumptions we will use throughout. We will also describe several geometric  quantities that are central to our analysis.
Section \ref{sec:performance} contains detailed analysis of the performance of algorithm of \citet{aanjaneya2012metric} for metric graph reconstruction, under modified settings and assumptions. 
In Section \ref{section::minimax}
we derive lower and upper bounds for the minimax risk of metric graph reconstruction problem.
In Section \ref{section::conclusion}
we conclude with some final comments.

\vspace{0.35cm}

{\em Related Work.}  The work most closely related to ours is
\citet{aanjaneya2012metric} which was, in fact, the motivation for our
work. From
the theoretical side, we replace the key assumption in
\citet{aanjaneya2012metric} of the sample being a
$(\varepsilon,R)$-approximation to the underlying metric graph, by the
milder assumption of the sample being dense in a neigborhood of the metric
graph. 
Approximation and reconstruction of metric graphs has also been considered in \cite{chazal2013gromov} and \cite{ge2011data}.
Metric graph reconstruction is related to the problem of estimating
stratified spaces (basically, intersecting manifolds).  Stratified
spaces have been studied by a number of authors such as
\cite{bendich2010towards, bendich2012local, bendich2008analyzing}.  A
spectral method for estimating intersecting structures is given in
\cite{arias2011spectral}.  There are a variety of algorithms for
specific problems, for example, see \cite{ahmed2012probabilistic,
  chen2010road} for the reconstruction of road networks.
Finally, \citet{chernov2013reconstructing}
derived an alternative algorithm that uses
ideas from homology.

\begin{figure}[ht!]
\centering
\includegraphics[scale=0.34]{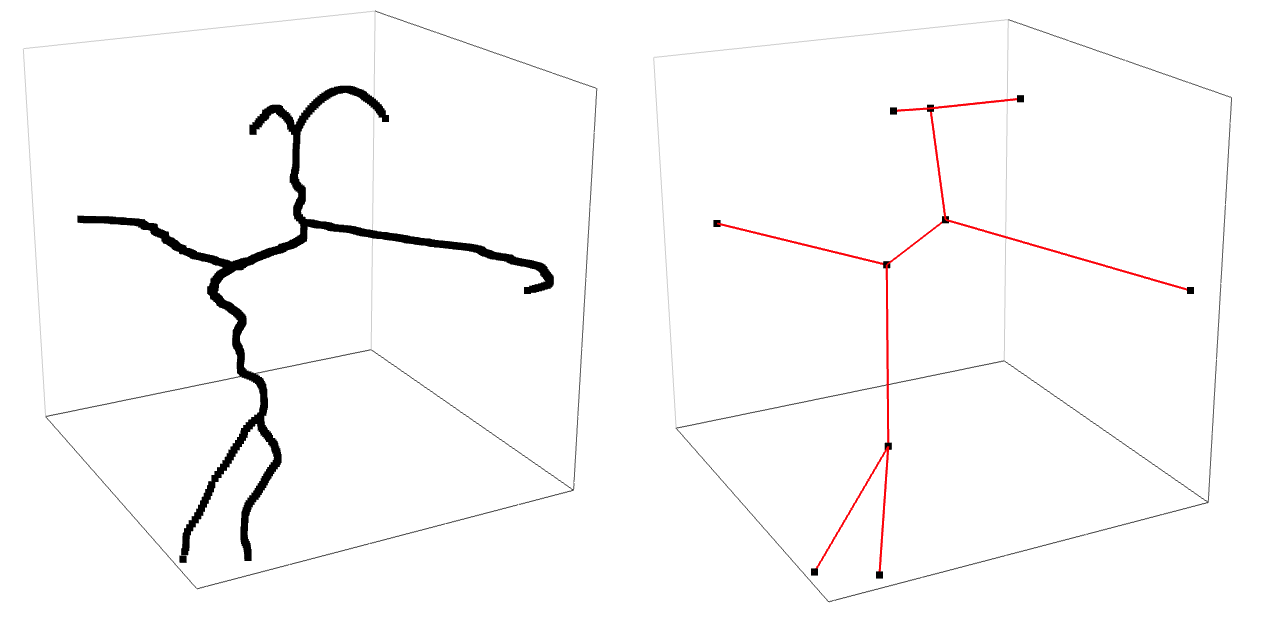}
\caption{Left: Neuron cr22e from the hippocampus of a rat; NeuroMorpho.Org \citep{ascoli2007neuromorpho}.
Right: A metric graph reconstruction of the neuron.}
\label{plot3D}
\end{figure}

\section{Background and Assumptions}
\label{section::background}

The assumptions in \citet{aanjaneya2012metric} lead to a  
reconstruction process that is aimed at capturing the intrinsic structure of the data and is somewhat oblivious to its extrinsic embedding. The authors assume that the sample comes with a metric that is close to
the intrinsic metric of the underlying graph, by imposing a limit on the Gromov-Hausdorff distance between the two metrics.
By considering data embedded in the Euclidean space and focusing on
the topological aspect, we show that the notion of \textit{dense sample} is sufficient
to guarantee a 
correct reconstruction.

In this section we provide background on metric graph spaces and
describe the assumptions and the geometric parameters that we will be
using throughout. \\
Informally, a metric graph is a collection of vertices and edges glued
together in some fashion. Here
we state the formal definitions of path metric space and metric
graph. For more details see \citet{aanjaneya2012metric} and
\citet{kuchment2004quantum}.

\begin{definition}
A metric space $(G, d_G)$ is a path metric space if the 
distance between any pair of points is equal to the infimum of the lengths of the continuous curves joining them.
A metric graph is a path metric space $(G, d_{G})$ that is homeomorphic to a 1-dimensional stratified space. 
A vertex of $G$ is a 0-dimensional stratum of $G$ and an edge of $G$ is a 1-dimensional stratum of $G$.
\end{definition}

We will consider metric graphs embedded in $\mathbb{R}^D$. Note that, if one ignores the metric structure, namely the length of edges and loops, the shape or topology of a metric graph $(G,d_G)$ is encoded by a graph, whose vertices and edges correspond to vertices and edges of $G$. Since we allow for two vertices to be connected by more than one edge we are actually dealing with pseudographs. 
We recall that an undirected pseudograph $(V,E)$ is a set of vertices $V$, a multiset $E$ of unordered pairs of (not necessarily distinct) vertices. To a given pseudograph we can associate a function $f : E \rightarrow V \times V$, which, when applied to an edge $e \in E$, simply extracts the vertices to which $e$ is adjacent. Thus, if $e_1,e_2 \in E$ are such that  $f(e_1) = f(e_2)$, then $e_1$ and $e_2$ are parallel edges. Similarly, if $e \in E$ is such that $f(e) = \{v,v\}$ for some $v \in V$, then $e$ is a loop. For each pair $(u,v) \in V \times V$, let $\nu(u,v) = |f^{-1}(\{u,v\})|$ if $\{u,v\}  \in E$ and $0$ otherwise. In particular, $\nu(u,v)$ is the number of edges between $u$ and $v$ (or loops if $u=v$). 

We say that a metric graph reconstruction algorithm perfectly recovers the
topology of $G$ if outputs a pseudograph isomorphic to the pseudograph representing the topology
of $G$.

We now define some key  quantities regarding the structure of a metric graph. 
We start with the definition of reach. 
Let $M$ be a 1-dimensional manifold embedded in
$\mathbb{R}^D$.
Let $T_uM$ denote the 1-dimensional tangent space to $M$ 
and let $T_u^\perp$M be the $(D-1)$-dimensional normal space.

\begin{definition}
Define the \textit{fiber}  of size $a$ at $u \in M$ to be $L_a(u,M)=T_u^\perp M \bigcap B(u,a)$, where $B(u,a)$ is the $D$-dimensional ball of radius $a$ centered at $u$. If $M$ has boundary $\{v_1,v_2\}$, the fiber of size $a$ at $v_i$ is defined as the limit of $L_a(u,M)$, as $u$ approaches $v_i$ in $M\backslash \{v_1,v_2\}$. The reach of $M$ is the largest number $\tau$ such that the fibers $L_\tau(u,M)$ never intersect. 
\end{definition}

The reach sets a limit on the curvature of a manifold. 
A manifold with large reach does not come too close to be
self-intersecting. For example the reach of an arc of a
circle is equal to its radius.
The quantity $1/\tau$ is called the \textit{condition number} in \cite{niyogi2008finding}. For more details see also
\citet{federer1959curvature, chazal2006topology, genovese2012minimax}. \\
Each edge of a metric graph $(G,d_G)$ can be seen as a 1-dimensional
manifold with boundary. 
Let the {\em local reach} of metric graph $G$ be the minimum reach associated to an edge of $G$. 

When 2 edges intersect at a vertex $v$ they
create an angle, 
where the angle between two intersecting curves
is formally defined as follows. 
Suppose that $e_1$ and $e_2$ intersect at $x$. Let $B(x,\epsilon)$ be the $D$-dimensional ball 
of radius $\epsilon$ centered at $x$. 
Let $\ell_1(\epsilon)$ be the line segment joining the two points
$x$ and $\partial B(x,\epsilon)\bigcap e_1$.
Let $\ell_2(\epsilon)$ be the line segment joining the two points
$x$ and $\partial B(x,\epsilon)\bigcap e_2$.
Let $\alpha_\epsilon(e_1,e_2)$ be the angle between 
$\ell_1(\epsilon)$ and $\ell_2(\epsilon)$.
The angle between $e_1$ and $e_2$ is 
$\alpha(e_1,e_2) = \lim_{\epsilon\to 0}\alpha_\epsilon(e_1,e_2)$.
We assume that, for each pair of intersecting edges $e_1$ and $e_2$,
the angle $\alpha(e_1,e_2)$ is well-defined.

To control points far away in the graph distance, but close in the embedding space, 
we define 
$$
A_G=\{(x,x') \in G\times G: d_{G}(x,x')\geq \min(b,\tau \alpha)\},
$$
where $b$ is the shortest edge of $G$, $\tau$ is the local condition number and $\alpha$ is the
smallest angle formed by two edges of $G$.
We define the {\em global reach} as the infimum of the Euclidean distances among 
pairs of point in $A_G$, that is  $\xi=\inf_{A_G}\Vert x-x'\Vert_2$.

Let $(G,d_G)$ be a metric graph and, for a constant $\sigma \geq 0$,
let $G_\sigma = \{y:\ \inf_{x\in G}||x-y||_2 \leq \sigma\}$ be the
$\sigma$-tube around $G$. If $\sigma=0$, then,
trivially, $G_\sigma = G$. Notice that $G_\sigma$ is a
set of dimension $D$ if $\sigma > 0$.

We will use the assumption that the sample $\mathbb{Y}$ is sufficiently dense
in $G_\sigma$ with respect to the Euclidean metric, as formalized below.

\begin{definition}
The sample
$\mathbb{Y}=\{y_1,\dots, y_n \} \subset G_\sigma \subset \mathbb{R}^D $ is $\frac{\delta}{2}$-dense in
$G_\sigma $ if for every $x \in G_\sigma$, there exists a $y \in \mathbb{Y}$ such that $\Vert x-y \Vert_2 < \frac{\delta}{2}$.
\label{dense}
\end{definition}

The problem of metric graph
reconstruction consists of reconstructing a metric graph $G$ given a
dense sample $\{y_1, \dots, y_n\} = \mathbb{Y} \subset G_\sigma$
endowed with a distance $d_{\mathbb{Y}}$, which could be the
$D$-dimensional Euclidean distance or some more complicate notion of distance. 
If $\sigma = 0$ we say that the sample $\mathbb{Y}$
is noiseless, while if $\sigma>0$, we say that $\mathbb{Y}$ is a noisy
sample.

Throughout our analysis we restrict
the attention to metric graphs embedded in $\mathbb{R}^D$ that satisfy the following assumptions:
\begin{itemize}
\item[{\bf A1}]  The graphs have finite total length and are free of
nodes of degree $2$ (though they may contain vertices of degree $1$ or
$3$ and higher).  
\item[{\bf A2}] Each edge is a smooth embedded sub-manifold
of dimension 1, of length at least $b>0$ and with reach at least $\tau>0$.
\item[{\bf A3}] Each pair of intersecting edges forms a well-defined angle of size at least $\alpha>0$.
\item[{\bf A4}] The global reach is at least $\xi>0$.
\end{itemize}
Assumptions A1 and A2 allow us  to consider each edge 
of a metric graph as a single smooth curve. 
A3 and A4 are additional regularity conditions on the separation between different edges.
Assumptions similar to A1-A4 are common in the literature. 
For different regularity conditions that allow for corners within an edge see, for example, 
\cite{chazal2009sampling} and 
\citet{chen2010road}.

Let $\mathcal{G}$ be the set of metric graphs embedded in $\mathbb{R}^D$ that satisfy assumptions A1, A2, A3 and A4, involving the parameters $b$, $\alpha$, $\tau$, $\xi$. 
We consider two noise models:\\

\textit{Noiseless.} We observe data $Y_1, \dots, Y_n \sim P$, where $P \in \mathcal{P}$, a collection of probability distributions supported
over metric graphs $(G,d_G)$ in $\mathcal{G}$ having densities $p$ with
respect to the length of $G$ bounded from below by a
constant $a>0$.  \\

\textit{Tubular Noise.} We observe data $Y_1, \dots, Y_n \sim P_{G,\sigma}$ where $P_{G,\sigma}$ is uniform on the $\sigma$-tube $G_\sigma$. In this case we consider the collection $\mathcal{P}=\{P_{G,\sigma}: G \in \mathcal{G} \}$.\\

We are interested in bounding the minimax risk
\begin{equation}
R_n= \inf_{\hat G} \sup_{P \in \mathcal{P}} P^n \Big(\hat G \not\simeq G  \Big),
\end{equation}
where the infimum is over all estimators $\hat G$ of the topology of $(G, d_G)$, the supremum is over the class of distributions $\mathcal{P}$ for 
$\mathbb{Y}$ and  $\hat G \not\simeq G $ means that $\hat G$ and $G$ are not isomorphic. \\
In Section \ref{section::minimax} we will find lower and upper bounds for $R_n$ in the noiseless case and the tubular noise case.\\
We conclude this section by summarizing the many parameters and symbols involved in our analysis. See Table \ref{tab::param}.
\vspace{0.7cm}

\begin{table}[h!]
\begin{center}
\begin{tabular}{| c | c |}
\hline
{\bf Symbol} & {\bf Meaning} \\ \hline \hline
$(G,d_G)$ & metric graph  \\ \hline
$\alpha$ & smallest angle  \\ \hline
$b$ & shortest edge \\ \hline
$\tau$ & local reach \\ \hline
$\xi$ & global reach \\ \hline
$\mathcal{G}$ & set of metric graphs embedded in $\mathbb{R}^D$, satisfying A1-A4 \\ \hline
$\mathcal{P}$ & set of distributions on $G$ or $G_\sigma$ \\ \hline \hline
$G_\sigma$ & $\sigma$ tube around $G$ \\ \hline
$\mathbb{Y}$ & sample, subset of $G_\sigma$ \\ \hline
$\delta$ & $\mathbb{Y}$ is a $\delta/2$-dense sample \\ \hline
\end{tabular}
\end{center}
\caption{Summary of the symbols used in our analysis.}
\label{tab::param}
\end{table}

\section{Performance Analysis for the Algorithm of \citet{aanjaneya2012metric}}
\label{sec:performance}
In this section we study the performance of the metric graph reconstruction algorithm of \citet{aanjaneya2012metric}, under assumptions A1-A4 and with a choice of parameters adapted to our setting. In Section \ref{section::minimax} we will use these results to derive bounds on the minimax rate for topology reconstruction. The metric graph reconstruction algorithm is presented in Algorithm \ref{alg::MGR}.

\begin{algorithm}[!hb]
Input: sample $\mathbb{Y},d_{\mathbb{Y}}, r, p_{11}$.\\
\small
1: \textbf{Labeling points as edge or vertex}\\
2: for all $y \in \mathbb{Y}$ do\\
3: \hspace{0.5cm} $S_y \leftarrow B(y,r+\delta)\backslash B(y,r)$\\
4: \hspace{0.5cm} $\text{deg}_r(y) \leftarrow $ Number of connected components of Rips-Vietoris graph $\mathcal{R}_{\delta}(S_y)$\\
5: \hspace{0.5cm} if $\text{deg}_r(y)=2$ then\\
6: \hspace{0.5cm} \hspace{0.5cm} Label $y$ as a edge point\\
7: \hspace{0.5cm} else\\
8: \hspace{0.5cm} \hspace{0.5cm} Label $y$ as a preliminary vertex point.\\
9: \hspace{0.5cm} end if\\
10: end for.\\
11: Label all points within Euclidean distance $p_{11}$ from a preliminary vertex point as vertices.\\
12: Let $\mathbb{E}$ be the point of $\mathbb{Y}$ labeled as edge points.\\
13: Let $\mathbb{V}$ be the point of $\mathbb{Y}$ labeled as vertices.\\ 
14: \textbf{Reconstructing the graph structure}\\
15: Compute the connected components of the Rips-Vietoris graphs $\mathcal{R}_\delta(\mathbb{E})$ and $\mathcal{R}_\delta(\mathbb{V})$.\\
16: Let the connected components of $\mathcal{R}_\delta(\mathbb{V})$ be the vertices of of the reconstructed graph $\hat G$.\\
17. Let there be an edge between vertices of $\hat G$ if their corresponding connected components
in $\mathcal{R}_\delta(\mathbb{V})$ contain points at distance less than $\delta$ from the same connected component of $\mathcal{R}_\delta(\mathbb{E})$.\\
Output: $\hat G$.
\caption{Metric Graph Reconstruction Algorithm}
\label{alg::MGR}
\end{algorithm}

The algorithm takes a (possibly noisy) sample $\mathbb{Y}$ from a
metric graph $G$ and a distance $d_\mathbb{Y}$ defined on $\mathbb{Y}$
and returns a graph $\hat G$ that approximates $G$.
The key idea is the following: a shell of radius $r$ is constructed around each point in the sample, which is labeled \textit{edge point} if its shell contains 2 well separated clusters of sampled points and \textit{vertex point} otherwise. Several steps of the algorithm require the construction of a Rips-Vietoris graph of parameter $\delta$: $\mathcal{R}_\delta(S_y)$ is a graph whose vertices are all the points of $S_y$ and there is an edge between two points if the Euclidean distance between them is not larger than $\delta$.
At Step 11 some of the edge points that are close to vertices are re-labeled as vertex points. This expansion guarantees a precise borderline between clusters of vertex points and clusters of edge points. At steps 15-17 each of these clusters is associated to a vertex or to an edge of the reconstructed graph $\hat G$.
We will analyze the algorithm considering the Euclidean distance on the
sample $\mathbb{Y}$, that is, $d_\mathbb{Y}= \Vert \cdot \Vert_2$.
The inner radius of the shell at Step 3 and the width of the
expansion at Step 11 are parameters the user has to specify. 

Before finding how dense a sample has to be in orderer to guarantee a correct reconstruction of a metric graph, we show that it is sufficient to study a particular metric graph embedded in $\mathbb{R}^2$, which represents the worst case. In other words, if the metric graph algorithm can reconstruct this particular planar graph, then it can reconstruct any other metric graph that satisfies A1-A4.

\subsection{The worst case: a metric graph in $\mathbb{R}^2$}
\label{sec::worst}
The worst case is the one for which it is hard to distinguish two edges that intersect at a vertex because they are too close in the embedding space.\\
Figure \ref{fig::worst} (top left) shows an edge $e$ that intersects two edges $e_1,e_2$ with reach $\tau$, forming an angle $\alpha$ at vertex $x$. For simplicity, we consider this metric graph embedded in $\mathbb{R}^3$ ($D=3$). Therefore Figure \ref{fig::worst} shows the projections of $e, e_1$ and $e_2$ on the (limit) plane formed by $e_1$ and $e_2$, passing through $x$.

\begin{figure}[ht!]
\centering
\includegraphics[scale=0.82]{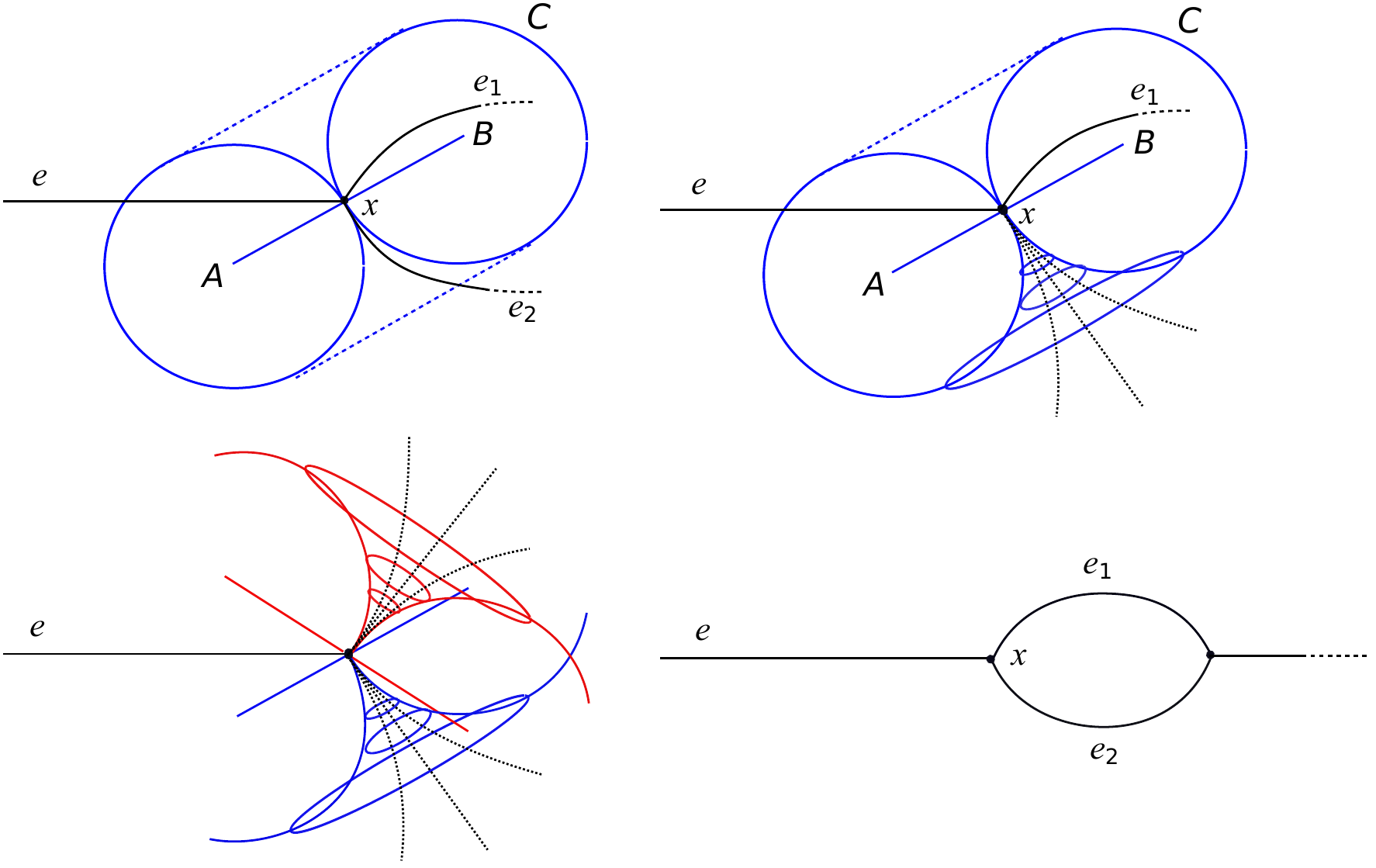}
\caption{Even in the worst case, edges $e_1$ and $e_2$ must lie outside of the torii constructed on the fibers $L_\tau(x,e_1)$ and $L_\tau(x,e_2)$.}
\label{fig::worst}
\end{figure}

We focus on edge $e_2$. 
The blue segment $\overline{AB}$ is the projection of $L_\tau(x,e_2)$, the fiber of size $\tau$ around $x$. In $\mathbb{R}^3$, $L_\tau(x,e_2)$ is a circle of radius $\tau$ centered in $x$.
By definition, for any $y \in e_2$, the fiber $L_\tau(y,e_2)$ can not intersect the fiber $L_\tau(x,e_2)$, otherwise the assumption involving the reach would be violated. We represent this condition by taking a circle $C$ of radius $\tau$ centered at $B$ and rotating it around $x$ along the circumference of $L_\tau(x,e_2)$. This procedure forms a torus with an inner loop of radius 0. Edge $e_2$ must lie outside of this torus, so that its fibers do not intersect $L_\tau(x,e_2)$.  See the top right plot of Figure \ref{fig::worst}.

The same reasoning applies to edge $e_1$, which must lie outside of the torus constructed on $L_\tau(x,e_1)$. See the bottom left plot. The worst case is the one for which $e_1$ and $e_2$ are as close as possible: on the same plane and on the boundaries of the two tori. This case is represented in the bottom right plot of Figure \ref{fig::worst}.
Note that $e_1$ and $e_2$ are simply arcs of circles of
radius $\tau$.

\begin{figure}[h]
\centering
\includegraphics[scale=0.82]{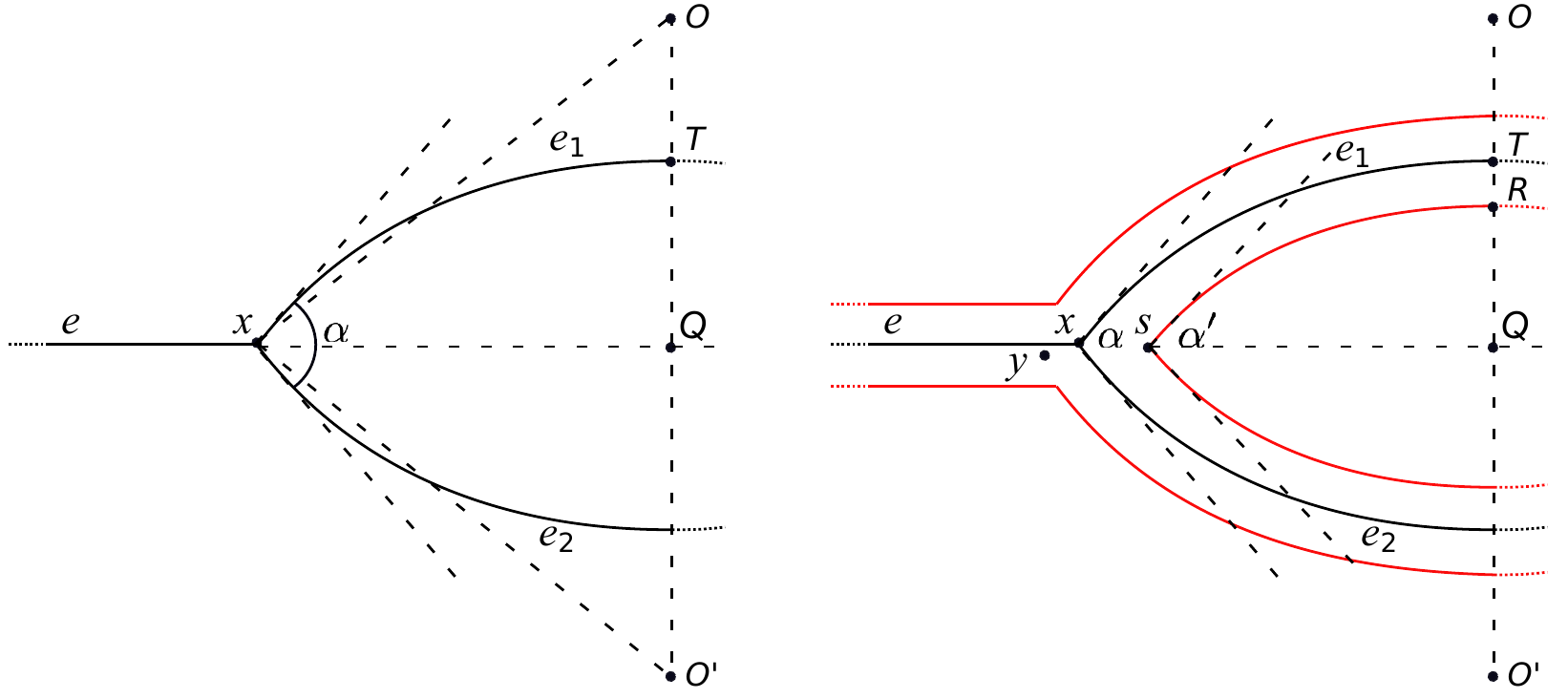}
\caption{Left: edges $e_1$ and $e_2$ with minimum reach $\tau$ forming the smallest 
angles $\alpha$ at vertex $x$. Right: same metric graph with a tube of radius $\sigma$ around it.}
\label{plot1aa}
\end{figure}

We will use basic trigonometric properties of the worst case.
In Figure \ref{plot1aa} (left), $O$ and $O'$ are the centers of the circles associated
to edges $e_1$ and $e_2$. 
It is easy to see that angle $O\hat x O'$
has width $\pi-\alpha$.
It can be
shown that
\begin{align}
x\hat O O' &= \alpha/2, \label{trick1}\\
T\hat x Q&=\alpha/4. \label{trick2}
\end{align}
Let  $\mathbb{Y}$ be a noisy sample of $G$. In other words $\mathbb{Y}$ is a subset of $G_\sigma$, the 
tube of radius $\sigma \geq 0$ around the metric graph $G$. See Figure \ref{plot1aa} (right).
Let $Q$ be the midpoint of segment $\overline{OO'}$ and let $T$ be the
intersection point of $\overline{OO'}$ and edge $e_1$.  
For $0 \leq \sigma \leq \overline{QT}= \tau-\tau \cos(\alpha/2)$, the smallest angle formed by the inner faces of
the tube around the metric graph is 
\begin{equation}
\label{eq::alphaPrime}
\alpha'= \pi - \arccos
\frac{2(\tau-\sigma)^2 - 4 \tau^2 \cos^2(\alpha/2)}{2 (\tau-\sigma)^2},
\end{equation}
where we applied the cosine law to the triangle $OsO'$ and the fact
that angle $O\hat s O'$ has width $\pi-\alpha'$.
Note that if $\sigma=0$ then $\alpha'= \alpha$.
As in \eqref{trick2}, it can be shown that 
\begin{equation}
\label{trick3}
R\hat s Q= \alpha' /4.
\end{equation}

The few basic trigonometric equations described above will be used to determine under which conditions on $b, \alpha, \tau, \xi, \sigma$ the metric graph reconstruction algorithm can reconstruct the worst case.

\subsection{Analysis of Algorithm \ref{alg::MGR} with Euclidean distance}
In this section we analyze Algorithm \ref{alg::MGR}. It is sufficient to study the worst case of figures \ref{fig::worst} and \ref{plot1aa} and extend the results to any metric graph in $\mathbb{R}^D$.
The Euclidean distance is used at every step of the algorithm, which requires the specification of $r$, the inner radius of the shell, and
$p_{11}$, the parameter governing the expansion of Step 11. We set
\begin{equation}
\label{eq::inner}
r= \frac{\delta}{2}+\sigma + \tau \sin(\alpha/2) - (\tau-\sigma) \sin(\alpha'/2) + \frac{\delta}{2 \sin(\alpha'/4)}
\end{equation}
and
\begin{equation}
\label{eq::p11}
p_{11}=   \frac{\delta}{2} + \tau \sin(\alpha/2) - (\tau-\sigma) \sin(\alpha'/2)+\frac{r+\delta}{\sin(\alpha'/2)} 
\end{equation}
This choice is justified in the proof of Proposition \ref{propAlg}.\\
Define
\begin{align}
&f(b,\alpha,\tau,\xi,\sigma) \nonumber :=\\
& \frac{(\tau-\sigma)\sin\left(\frac{\min(b, \alpha\tau) - (\alpha-\alpha')\tau}{2\tau}\right) - \left[ \tau \sin(\alpha/2) - (\tau-\sigma) \sin(\alpha'/2)\right] \left(1+\frac{2}{\sin(\alpha'/2)} \right) -  \frac{2\sigma}{\sin(\alpha'/2)} }{1+3[\sin(\alpha'/2)]^{-1}+[\sin(\alpha'/2)\sin(\alpha'/4)]^{-1} },\label{def::f}
\end{align}
where $\alpha'$ is given in \ref{eq::alphaPrime}.
Note that $f(b,\alpha,\tau,\xi,\sigma) $ is a decreasing function of $\sigma$.

\begin{proposition}
If  $\mathbb{Y}$ is $\frac{\delta}{2}$-dense in $G_\sigma$ and 


\begin{equation}
0 < r+\delta< \xi - 2\sigma,
\label{cond1}
\end{equation}
\begin{equation}
0 < \delta < f(b,\alpha,\tau,\xi,\sigma),
\label{cond2}
\end{equation}
then the graph $\hat G$ provided by Algorithm 1 (input: $\mathbb{Y},\Vert \cdot \Vert_2, r, p_{11}$) is isomorphic to $G$.
\label{propAlg}
\end{proposition}

\begin{proof} 
We will show that under conditions \eqref{cond1} and \eqref{cond2}, Algorithm 1 can reconstruct the worst case described in section \ref{sec::worst}, formed by edges $e_1$ and $e_2$ of reach $\tau$ forming an angle of width $\alpha$. This will automatically imply that the algorithm can reconstruct the topology of other vertices and edges in the $D$-dimensional space.

Condition \eqref{cond1} guarantees that points of $G$ which are
far apart in the metric graph distance $d_G$, and close in the
embedding space, do not interfere in the construction of the shells at Steps 3-4. 

The rest of the proof involves condition
\eqref{cond2}. 
Since the sample is $\frac{\delta}{2}$-dense in the
tube, there is at least a point $y \in \mathbb{Y}$ inside the ball of
radius $\frac{\delta}{2}$ centered at any vertex $x \in G$. When using Algorithm 1 we want to be sure that $y$ is labeled as a vertex,
that is, the number of connected components of the shell around $y$ is different than 2 
(Steps 3-4). The worst case is depicted in Figure
\ref{plot2b} (left), where $x$ is the vertex of minimum angle
$\alpha$, formed by two edges, $e_1$ and $e_2$ of reach $\tau$.
First, we show that for the the value of $r$ selected in \eqref{eq::inner}, points close to an actual vertex are labeled as vertices at Steps 3-10 and points far from actual vertices are labeled as edges. 
\begin{figure}[h]
\centering
\includegraphics[scale=0.82]{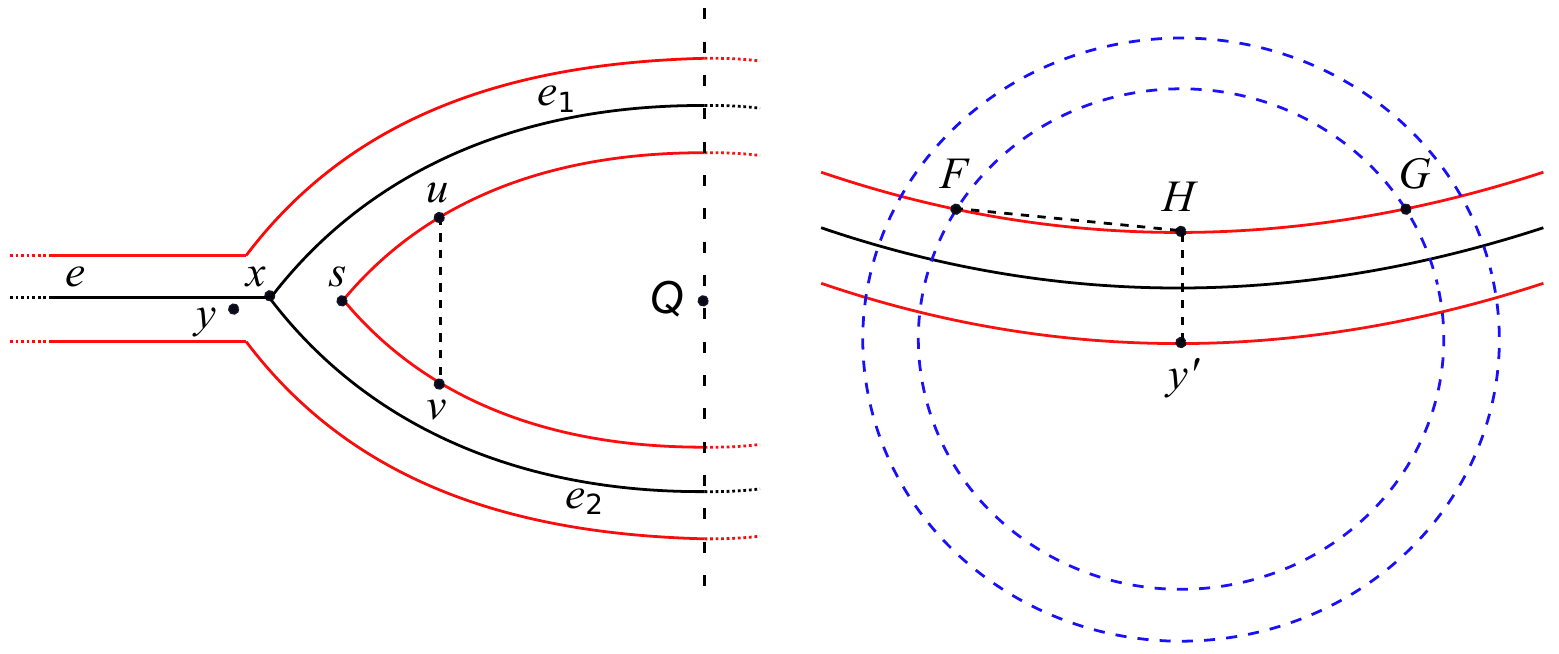}
\caption{Left: edges $e_1$ and $e_2$ with minimum 
reach $\tau$ forming the smallest angles $\alpha$ at vertex $x$. 
Right: The distance $\Vert F-G \Vert_2$ between the two connected components of the shell around an edge point $y'$ must be greater than $\delta$.}
\label{plot2b}
\end{figure}
The inner faces of the tube of radius $\sigma$ around $e_1$ and $e_2$
form an angle of width $\alpha'$ at vertex $s$, as described in Section \ref{sec::worst}.
Let $u$ and $v$ be the two points on the faces of the tube such that they are equidistant from $x$ and $\Vert u-v\Vert_2=\delta$. 
Since at Step 4 we construct a $\delta$-graph to determine the number of connected components of the 
shell $S_y$ and we want $y$ to be a vertex, we choose $r$, the inner radius of the shell $S_y$, 
so that if $u,v \in \mathbb{Y}$ then $r \geq \max\{d_\mathbb{Y}(y,u), d_\mathbb{Y}(y,v)\}$. 
This guarantees that $\forall t_1,t_2 \in \mathbb{Y}$ with $t_1$ around edge $e_1$, $t_2$ 
around edge $e_2$ such that $\{t_1,t_2 \}\subset S_y$, we have 
$d_\mathbb{Y}(t_1,t_2)\geq \delta$, that is $t_1$ and $t_2$ belong to different connected components of the shell around $y$ at Step 4.\\
The distance between $y$ and $u$ is bounded by 
$ \Vert y-x\Vert_2 +\Vert x-s \Vert_2 + \Vert s-u \Vert_2$, where, using
\eqref{trick1},
$$
\Vert x-s \Vert_2 = \Vert x-Q \Vert_2 - \Vert s-Q \Vert_2 = \tau \sin(\alpha/2) - (\tau-\sigma) \sin(\alpha'/2)
$$
and using \eqref{trick3}, 
\begin{equation}
\Vert s-u \Vert_2 \leq \frac{\delta}{2 \sin(\alpha'/4)} .
\label{su}
\end{equation}
Therefore we require that $r$, the inner radius of the shell of Step 4 satisfies
\begin{align}
r & \geq  \frac{\delta}{2} + \Vert x-s \Vert_2 +\frac{\delta}{2 \sin(\alpha'/4)} \label{eq::r1} \\
& \geq \Vert y-x\Vert_2 +\Vert x-s \Vert_2 + \Vert s-u \Vert_2. \nonumber
\end{align}
Another condition on $r$ arises when we label edge points far from actual vertices. See Figure \ref{plot2b} (right). If $y' \in \mathbb{Y}$, then it should be labeled as an edge point. That is, at Step 4, the Rips graph $\mathcal{R}_\delta(S_{y'})$ on the shell $S_{y'}$ should have 2 connected components. Therefore the distance $\Vert F-G \Vert_2$ between them must be greater than $\delta$. 
We require that
\begin{equation}
\label{eq::r2}
r \geq 2\sigma + \delta/ \sqrt{2}
\end{equation}  
which implies $\Vert F-G \Vert_2 > \delta $ when $r$ is small enough, as implied by \eqref{cond2}.\\
Note that the value $r= \frac{\delta}{2}+\sigma + \Vert x-s\Vert_2 + \frac{\delta}{2 \sin(\alpha'/4)}$ satisfies both \eqref{eq::r1} and \eqref{eq::r2}.

The outer radius of the shell at Steps 3-4 has length $r+\delta$. This guarantees that when the shell around an edge point intersects the tube around $G$ there is at least a point $y \in \mathbb{Y}$ in each connected component of the shell, since $\mathbb{Y}$ is $\frac{\delta}{2}$-dense in $G_\sigma$.

In the last part of this proof we show that condition \eqref{cond2} is needed to guarantee that the sample is dense enough and the radius of the shells of Step 3 has the correct size, so that, even in the worst case, each vertex is associated to one set of sampled points at Steps 15-17 and these connected components are correctly linked by sets of sampled points labeled as edge points.\\
Let $z\in G_\sigma$ be the point around $e_2$ where the segment of length $r+\delta$, orthogonal to the face of the tube around edge $e_1$, intersects the face of the tube around edge $e_2$. See Figure \ref{plot2c}.
If this segment does not exist we simply consider the segment of length $r+\delta$ from $s$ to a point 
$z$ on $e_2$.

\begin{figure}[h]
\centering
\includegraphics[scale=0.82]{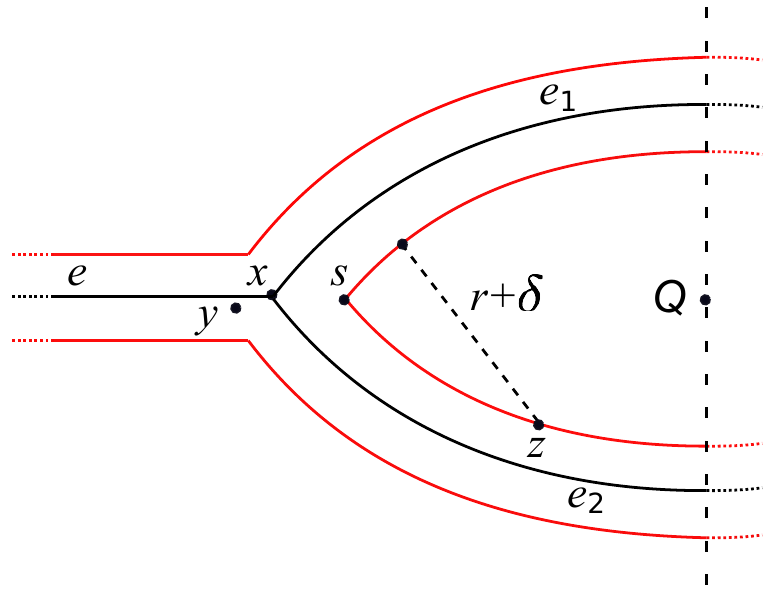}
\caption{The shell around $z$ is tangent to edge $e_2$. }
\label{plot2c}
\end{figure}

Suppose $z \in \mathbb{Y}$. Among the points that might be labeled as vertices at Step 6 because of their closeness to vertex $x$, $z$ is the furthest from $x$, since the shell around $z$ is tangent to the tube around $e_1$.
At Step 11, in order to control the labelling of the points in the tube between $y$ and $z$ we would like to label all the points in $\{y' \in \mathbb{Y}: \Vert y'-y \Vert_2 \leq  \Vert y-z \Vert_2\}$ as vertices.
To simplify the calculation we use the following bound
\begin{align*}
\Vert y-z \Vert_2 &\leq \Vert y-x \Vert_2 + \Vert x-s \Vert_2 + \Vert s-z \Vert_2 , 
\end{align*}
where, using \eqref{trick3},
\begin{equation}
\Vert s-z\Vert_2 \leq \frac{r+\delta}{\sin(\alpha'/2)} .
\label{sz}
\end{equation}
This justifies the choice of $\displaystyle p_{11}= \frac{\delta}{2} + \Vert x-s\Vert_2+ \frac{r+\delta}{\sin(\alpha'/2)} \geq \Vert y-z\Vert_2  $. Thus, at Step 11 we label all the points in $\{y' \in \mathbb{Y}: \Vert y'-y \Vert_2 \leq  p_{11} \text{ and } y \text{ is labeled as vertex at Step 6 }\}$ as vertices.
If $z$ is actually labeled as a vertex at Step 6, then through the expansion of Step 11, all the points at distance not greater than $p_{11}$ from $z$ are labeled as vertices.\\
Finally we determine under which conditions there is at least a point in the tube around $e_2$ labeled as an edge point after Step 11. Consider the worst case in which $e_1$ and $e_2$ are forming an angle of size $\alpha$ at both their extremes $x$ and $x'$. See Figure \ref{plot3b}.

\begin{figure}[h]
\centering
\includegraphics[scale=0.82]{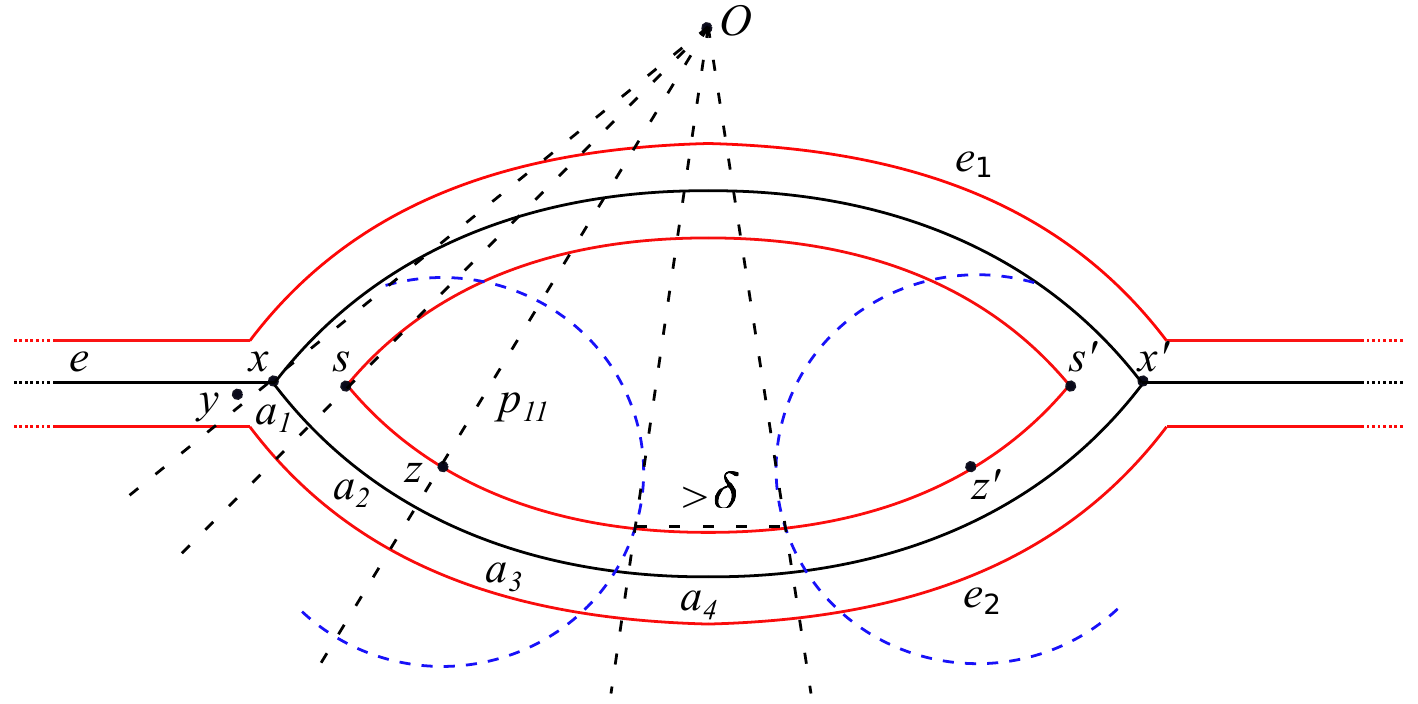}
\caption{Edges $e_1$ and $e_2$, forming an angle of size $\alpha$ at both their extremes $x$ and $x'$.}
\label{plot3b}
\end{figure}

\noindent All the points $y' \in \mathbb{Y}$ such that $\Vert y'-z \Vert_2  \leq p_{11}$ or $\Vert y'-z' \Vert_2  \leq p_{11}$ might be labeled as vertices. 
When we construct $\mathcal{R}(\mathbb{E})_{\delta}$ and $\mathcal{R}(\mathbb{V})_{\delta}$ at Step 15 the two sets of vertices around $x$ and $x'$ must be disconnected and there must be at least an edge point between them. A sufficient condition is that the length of edge $e_2$ is greater than $2(a_1+a_2+a_3)+a_4$, where
\begin{itemize}
\item $a_1$ is the length of the arc of $e_2$ formed by the projections of lines $\overline{Ox}$ and $\overline{Os}$ on $e_2$,
\item $a_2$ is the length of the arc of $e_2$ formed by the projection of the chord of length $\Vert s-z \Vert_2$,
\item $a_3$ is the length of the arc of $e_2$ formed by the projection of the chord of length $p_{11}$,
\item $a_4$ is the length of the arc of $e_2$ formed by the projection of the chord of length $\delta$.
\end{itemize}
Note that, in Figure \ref{plot3b}, $e_2 = 2 \tau \arcsin\left(\frac{\Vert x-x'\Vert_2}{2 \tau} \right)=\alpha\tau$ but in general it might be shorter, so that $e_1$ and $e_2$ might not intersect in $x'$. However, by assumptions A2, $e_2$ must be longer than $b$. Thus we require
\begin{equation}
\min(b, \alpha\tau) > 2(a_1+a_2+a_3)+a_4.
\label{condition}
\end{equation}
By simple properties involving arcs and chords we have
\begin{align*}
a_1&= \left(\frac{\alpha-\alpha'}{2} \right) \tau, \hspace{2cm}
a_2= 2\tau \arcsin\left( \frac{\Vert s-z \Vert_2}{2(\tau-\sigma)} \right),\\
a_3&= 2\tau \arcsin\left( \frac{p_{11}}{2(\tau-\sigma)} \right), \quad
a_4= 2\tau \arcsin\left( \frac{\delta}{2(\tau-\sigma)} \right).
\end{align*}

\noindent Since the arcsin is superadditive in $[0,1]$ we require the stronger condition
$$
\min(b, \alpha\tau) - (\alpha-\alpha')\tau > 2\tau \arcsin\left(  \frac{2 \Vert s-z\Vert_2 +2 p_{11} + \delta}{2(\tau-\sigma)} \right),
$$
which holds if
$$
\sin\left(\frac{\min(b, \alpha\tau) - (\alpha-\alpha')\tau}{2\tau}\right) > \frac{2 \frac{r+\delta}{\sin(\alpha'/2)}  +2 p_{11} + \delta}{2(\tau-\sigma)}.
$$
The last condition is equivalent to \eqref{cond2}.
If this condition is satisfied then the graph is correctly reconstructed at Steps 15-17: every connected component of $\mathcal{R}_\delta(\mathbb{V})$ corresponds to a vertex of $G$ and every connected component of $\mathcal{R}_\delta(\mathbb{E})$ corresponds to an edge of $G$. 
\end{proof}

\begin{example}
\textbf{A Neuron in Three-Dimensions.} 
We return to the neuron example and we try to apply Propositions \ref{propAlg} to the 3D data of 
Figure \ref{plot3D}, namely the neuron cr22e from the
hippocampus of a rat \citep{gulyas1999total}. The data were obtained
from NeuroMorpho.Org \citep{ascoli2007neuromorpho}.  The total length
of the graph is $1750.86  \mu m$. We assume the smallest edge has
length $100 \mu m$, the smallest angle $\pi/3$, the local reach $30 \mu m$ and $\xi=50 \mu m$.  The conditions of Proposition \ref{propAlg}
are satisfied for $\delta=2.00 \mu m$. Algorithm \ref{alg::MGR} reconstructs the topology of the metric graph starting from a $\delta/2$-dense sample.
Figure \ref{plot3D}b shows the reconstructed
graph.  
\end{example}

\section{Minimax Analysis}
\label{section::minimax}

In this section we derive lower and upper bound for the minimax risk
\begin{equation}
R_n= \inf_{\hat G} \sup_{P \in \mathcal{P}} P^n \Big(\hat G \not\simeq G  \Big),
\end{equation}
where, as described in Section \ref{section::background}, the infimum is over all estimators $\hat G$ of the metric graph $G$, the supremum is over the class of distributions $\mathcal{P}$ for 
$\mathbb{Y}$ and  $\hat G \not\simeq G $ means that $\hat G$ and $G$ are not isomorphic. 

\subsection{Lower Bounds}
To derive a lower bound on the minimax risk, we make repeated use of Le Cam's lemma. See, e.g.,
\cite{yu1997assouad} and Chapter 2 of \citet{tsybakov2008introduction}.  Recall that the total variation distance
between two measures $P$ and $Q$ on the same probability space is defined by TV$(P,Q)=\sup_A|P(A)-Q(A) |$ where the supremum is over all measurable
sets. It can be shown that TV$(P,Q)=P(H)-Q(H)$, where $H=\{y:p(y)\geq q(y) \}$ and $p$ and $q$ are the densities of $P$ and $Q$ with respect to any measure that dominates both $P$ and $Q$.

\begin{lemma}[\textbf{Le Cam}]
Let $\mathcal{Q}$ be a set of distributions. Let $\theta(Q)$ take
values in a metric space with metric $\rho$. Let $Q_1,Q_2 \in
\mathcal{Q}$ be any pair of distributions in $\mathcal{Q}$. Let
$Y_1,\dots,Y_n$ be drawn iid from some $Q \in \mathcal{Q}$ and
denote the corresponding product measure by $Q^n$. Then
\begin{equation}
\inf_{\hat\theta} \sup_{Q \in \mathcal{Q}}
\mathbb{E}_{Q^n}\left[\rho(\hat \theta, \theta(Q)) \right] \geq
\frac{1}{8} \rho(\theta(Q_1), \theta(Q_2)) (1-
\text{TV}(Q_1,Q_2))^{2n}
\end{equation}
where the infimum is over all the estimators of $\theta(Q)$.
\end{lemma}

Below we apply Le Cam's lemma using
several pairs of distributions. Any pair $Q_1,Q_2$ is associated
with a pair of metric graphs $G',G'' \in \mathcal{G}$. We take
$\theta(Q_1)$ and $\theta(Q_2)$ to be the classes of graphs that are
isomorphic to $G'$ and $G''$. We set $\rho(\theta(Q_1),
\theta(Q_2))=0$ if $G'$ and $G''$ are isomorphic and
$\rho(\theta(Q_1), \theta(Q_2))=1$ otherwise.
Figure \ref{lower} shows several pairs of metric graphs that are used
to derive lower bounds in the 
noiseless case and in the tubular noise case. 
In the noiseless case we ignore the $\sigma$-tubes around the metric graphs.

\begin{figure}[h]
\centering
\includegraphics[scale=0.85]{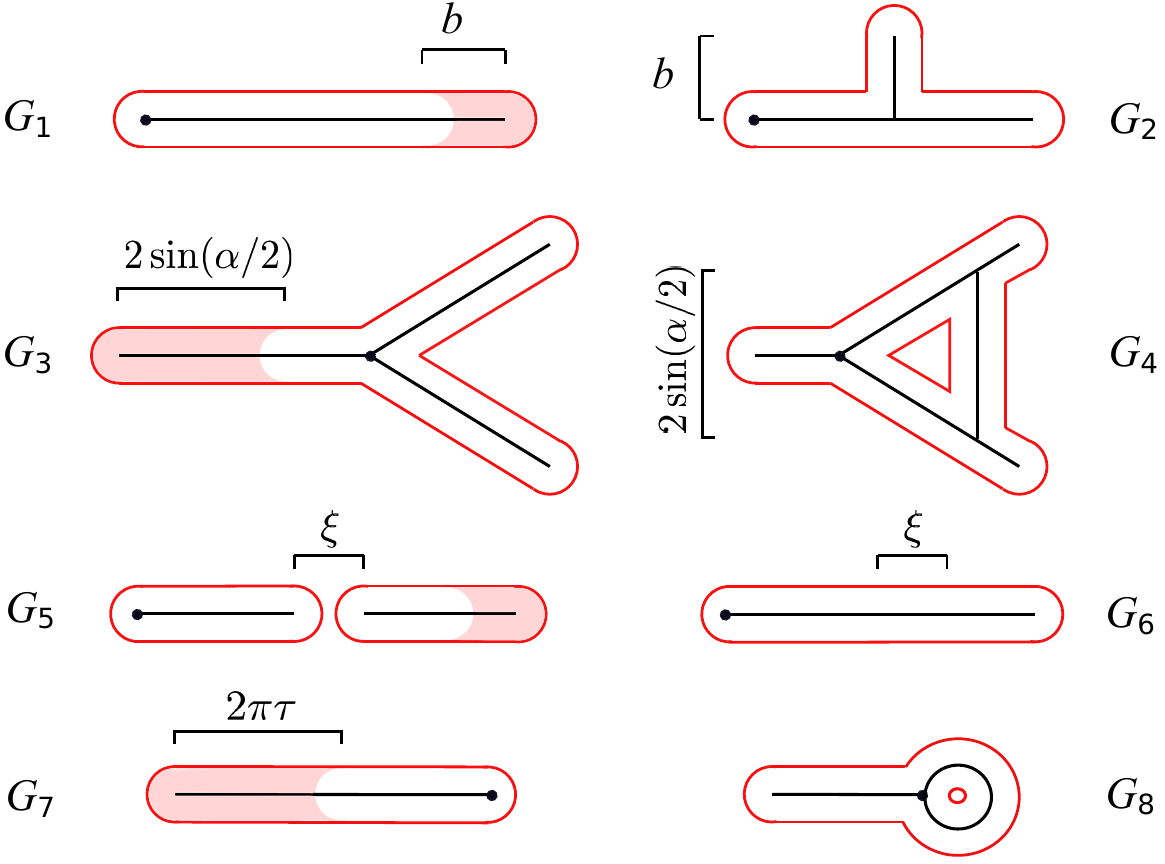}
\caption{Pairs of metric graphs used in the derivation of lower bounds in the noiseless case and in the tubular noise case.}
\label{lower}
\end{figure}

\begin{theorem}
\label{th::lowerNoiseless}
In the noiseless case ($\sigma=0$), for $b\leq b_0(a)$, 
$\alpha\leq \alpha_0(a)$, $\xi\leq \xi_0(a)$, $\tau\leq \tau_0(a)$, where $b_0(a), \alpha_0(a), \xi_0(a)$ and $\tau_0(a)$ 
are constants which depend on $a$, a lower bound on the minimax risk for metric graph reconstruction is
\begin{equation}
R_n \geq 
\exp\Bigl( -2a \min\{b, 2 \sin(\alpha/2), \xi, 2\pi \tau \} n \Bigr).
\end{equation}
\end{theorem}

\begin{proof}  We consider the 4 parameters separately. See Figure \ref{lower}, ignoring the red lines representing the tubular noise that is not considered in this theorem.

{\it Shortest edge $b$}. Consider the metric graph $G_1$ consisting of a single edge of length 1+$b$ and metric graph $G_2$ with an edge of length 1 and an orthogonal edge of length $b$ glued in the middle.
The density on $G_1$ is constructed in the following way: on the set $G_1 \backslash G_2$ of 
length $b$ we set $p_1(x)=a$ and the rest of the mass is evenly distributed over the remaining portion of $G_1$. 
Similarly, for $G_2$ we set $p_2(x)=a$ on $ G_2 \backslash G_1$, which correspond to the orthogonal edge of length $b$. 
We evenly spread the remaining mass. 
The two densities differ only on the sets $G_1 \backslash G_2$ and $G_2 \backslash G_1$. Therefore
$
\text{TV}(p_1,p_2) \leq a b
$ 
and, by Le Cam's lemma, 
$
R_n \geq \frac{1}{8} (1- ab)^{2n} \geq \frac{1}{8} \text{e}^{-2abn}
$ 
for all $b\leq b_0(a)$, where $b_0(a)$ is a constant depending on $a$.

 {\it Smallest angle $\alpha$}. Now consider the metric graphs $G_3$ and $G_4$. 
$G_3$ consists of two edges of length 2 forming an angle $\alpha$ and a 
third edge of length $1+2 \sin(\alpha/2)$ glued to the first two. 
$G_4$ is similar: an edge of length $2\sin(\alpha/2)$ is added to complete the triangle, while the edge on the left has length 1. As in the previous case we set $p_3(x)=a$ on $G_3\backslash G_4$,  $p_4(x)=a$ on  $G_4\backslash G_3$ and spread evenly the rest of the mass.
The total variation distance is 
$
\text{TV}(p_3,p_4)\leq 2a \sin\left(\frac{\alpha}{2}\right) 
$ 
and, by Le Cam's lemma, 
$
R_n \geq \frac{1}{8} (1- 2a \sin\left(\alpha/2 \right) )^{2n} \geq \frac{1}{8} \text{e}^{-4a\sin\left(\alpha/2 \right)n}
$ 
for all $\alpha\leq \alpha_0(a)$, where $\alpha_0(a)$ is a constant depending on $a$.

 {\it Global reach $\xi$}. We defined the global reach as the shortest euclidean distance between two 
points that are far apart in the graph distance.
Figure \ref{lower} shows metric graph $G_5$ formed by a single edge of length 1 and 
metric graph $G_6$ consisting of two edges of length $0.5$, $\xi$ apart from each other.
Again, we set $p_5(x)=a$ on $G_5 \backslash G_6$, $p_6(x)=a$ on $G_6 \backslash G_5$ and evenly spread the rest.
We obtain
$
\text{TV}(p_5,p_6)\leq a\xi
$ 
and, by Le Cam's lemma, 
$
R_n \geq \frac{1}{8} (1- a\xi)^{2n} \geq \frac{1}{8} \text{e}^{-2a\xi n}
$ 
for all $\xi\leq \xi_0(a)$, where $\xi_0(a)$ is a constant depending on $a$.
 
{\it Local reach $\tau$}. The local reach $\tau$ is the smallest reach of the edges forming the metric graph.
Consider metric graphs $G_7$ and $G_8$. $G_7$ consists of a loop of radius $\tau$ attached to an edge of length 1 and metric graph $G_8$ is a single edge of length $1+2 \pi \tau$.
As in the previous cases $p_7(x)=a$ on $G_7 \backslash G_8$ and $p_8(x)=a$ on $G_8 \backslash G_7$. 
It follows that 
$
\text{TV}(p_7,p_8)\leq 2a \pi \tau
$ 
and, by Le Cam's lemma, 
$
R_n \geq \frac{1}{8} (1- 2a \pi \tau)^{2n} \geq \frac{1}{8} \text{e}^{-4a \pi \tau n}
$ 
for all $\tau\leq \tau_0(a)$, where $\tau_0(a)$ is a constant depending on $a$. 
\end{proof}

For the tubular noise case we assume that $\sigma$ is small enough to guarantee that $R_n<1$, that is, the problem is not hopeless. In particular, we require that $\sigma$ satisfies conditions \eqref{cond1} and \eqref{cond2} of Proposition \ref{propAlg}, which can be combined into the following condition

\begin{equation}
0 < \min\left\{ \frac{\xi-3\sigma- \tau \sin(\alpha/2) + (\tau-\sigma) \sin(\alpha'/2)}{3/2 +[2 \sin(\alpha'/4)]^{-1}}, \,  f(b,\alpha,\tau,\xi,\sigma)  \right\}.
\label{cond:tube}
\end{equation}

\begin{theorem}
\label{th::lowerNoise}
Assume that $\sigma$ is positive and satisfies condition \eqref{cond:tube}.
In the tubular noise case, for $b\leq b_0(D)$, 
$\alpha\leq \alpha_0(D)$, $\xi\leq \xi_0(D)$, $\tau\leq \tau_0(D)$, where $b_0(D), \alpha_0(D), \xi_0(D)$ and $\tau_0(D)$ 
are constants which depend on the ambient dimension $D$, a lower bound on the minimax risk for metric graph reconstruction is
\begin{equation}
R_n \geq  \frac{1}{8} 
\exp\Bigl( -2 \min\{ C_{D,1}b, C_{D,2}\sin(\alpha/2), C_{D,3} \xi, C_{D,4} \tau    \} n \Bigr),
\end{equation}
for some constants $C_{D,1},C_{D,2},C_{D,3},C_{D,4}$.
\end{theorem}

\begin{proof}  As in the proof oh Theorem \ref{th::lowerNoiseless} we consider the 4 parameters separately. We compare the pairs of graphs shown in Figure \ref{lower}, including the tubular regions constructed around them, from which we get samples uniformly.

{\it Shortest edge $b$}. Consider the metric graph $G_1$ consisting of a single edge of length 1+$b$ 
and metric graph $G_2$ with an edge of length 1 and an orthogonal edge of length $b$ glued in the middle. 
Since $\text{vol}(G_1)>\text{vol}(G_2)$, the density $q_1$ at a point in the tube around $G_1$ is lower than the density $q_2$ at a point around $G_2$.
From the definition of total variation $TV=q_1(H)-q_2(H)$ where $H$ is the set where $q_1>q_2$, the shaded area in Figure \ref{lower}. Note that $q_2(H)=0$ and
$$
TV(q_1,q_2)=q_1(H)=\frac{\text{vol}(H)}{\text{vol}(G_1)} \leq C_{D,1} \frac{b \sigma^{D-1}}{(1+b)\sigma^{D-1}}\leq C_{D,1} b.
$$
By Le Cam's lemma, 
$
R_n \geq \frac{1}{8} (1- C_{D,1} b)^{2n} \geq \frac{1}{8} \text{e}^{-2 C_{D,1} bn}
$ 
for all $b\leq b_0(D)$, where $b_0(D)$ is a constant depending on $D$.

 {\it Smallest angle $\alpha$}. Now consider the metric graphs $G_3$ and $G_4$. 
Since $\text{vol}(G_3)>\text{vol}(G_4)$, the density $q_3$ at a point in the tube around $G_3$ is lower than the density $q_4$ at a point around $G_4$.
$TV=q_3(H)-q_4(H)$ where $H$ is the set where $q_3>q_4$, the shaded area in the tube around $G_3$. Note that $q_4(H)=0$ and
$$
TV(q_3,q_4)=q_3(H)=\frac{\text{vol}(H)}{\text{vol}(G_3)} \leq C_{D,2} \frac{\sin(\alpha/2) \sigma^{D-1}}{(1+\sin(\alpha/2))\sigma^{D-1}}\leq C_{D,2} \sin(\alpha/2).
$$
By Le Cam's lemma, 
$
R_n \geq \frac{1}{8} (1- C_{D,2} \sin(\alpha/2))^{2n} \geq \frac{1}{8} \text{e}^{-2 C_{D,2} \sin(\alpha/2) n}
$ 
for all $\alpha \leq \alpha_0(D)$, where $\alpha_0(D)$ is a constant depending on $D$.

 {\it Global reach $\xi$}. Figure \ref{lower} shows metric graph $G_5$ formed by a single edge of length 1 and 
metric graph $G_6$ consisting of two edges of length $0.5$, $\xi$ apart from each other.
Since $\text{vol}(G_5)>\text{vol}(G_6)$, the density $q_5$ at a point in the tube around $G_5$ is lower than the density $q_6$ at a point around $G_6$.
$TV=q_5(H)-q_6(H)$ where $H$ is the set where $q_5>q_6$, the shaded area in the tube around $G_5$. Note that $q_6(H)=0$ and
$$
TV(q_5,q_6)=q_5(H)=\frac{\text{vol}(H)}{\text{vol}(G_5)} \leq C_{D,3} \frac{ \xi \sigma^{D-1}}{\sigma^{D-1}} = C_{D,3} \xi.
$$
By Le Cam's lemma, 
$
R_n \geq \frac{1}{8} (1- C_{D,3} \xi)^{2n} \geq \frac{1}{8} \text{e}^{-2 C_{D,3} \xi n}
$ 
for all $\xi \leq \xi_0(D)$, where $\xi_0(D)$ is a constant depending on $D$.

{\it Local reach $\tau$}. The local reach $\tau$ is the smallest reach of the edges forming the metric graph.
Consider metric graphs $G_7$ and $G_8$ in Figure \ref{lower}. 
Since $\text{vol}(G_7)>\text{vol}(G_8)$, the density $q_7$ at a point in the tube around $G_7$ is lower than the density $q_8$ at a point around $G_8$.
$TV=q_7(H)-q_8(H)$ where $H$ is the set where $q_7>q_8$, the shaded area in the tube around $G_7$. Note that $q_8(H)=0$ and
$$
TV(q_7,q_8)=q_7(H)=\frac{\text{vol}(H)}{\text{vol}(G_7)} \leq C_{D,4} \frac{ \tau \sigma^{D-1}}{(1+\tau)\sigma^{D-1}} \leq C_{D,4} \tau.
$$
By Le Cam's lemma, 
$
R_n \geq \frac{1}{8} (1- C_{D,4} \tau)^{2n} \geq \frac{1}{8} \text{e}^{-2 C_{D,4} \tau n}
$ 
for all $\tau \leq \tau_0(D)$, where $\xi_0(D)$ is a constant depending on $D$.
\end{proof}

Note that, up to constants, the lower bound obtained in the tubular noise case is identical to the lower bound of Proposition \ref{th::lowerNoiseless} for the noiseless case.

\subsection{Upper Bounds}
In this section we use the analysis of the performance of Algorithm \ref{alg::MGR} to derive an upper bound on the minimax risk. 
We will use the strategy of \citet{niyogi2008finding} to find the sample size that guarantees a $\delta/2$-dense sample with high probability. We will use the following two lemmas. 

\begin{lemma} [\textbf{5.1 in \citet{niyogi2008finding}}]
\label{NSW1}
Let $\{A_i\}$ for $i=1,\dots,l$ be a finite collection of measurable sets and let $\mu$ be a probability measure on $\bigcup_{i=1}^l A_i$ such that for all $1\leq i \leq l$, we have $\mu(A_i)>\gamma$. Let $\bar x=\{x_1,\dots,x_n \}$ be a set of n i.i.d. draws according to $\mu$. Then if 
$$
n\geq \frac{1}{\gamma} \left(\log l + \log \left(\frac{1}{\lambda} \right) \right)
$$
we are guaranteed that with probability $> 1-\lambda$, the following is true:
$$
\forall i, \;\; \bar x \cap A_i \neq \emptyset.
$$
\end{lemma}

\noindent 
Recall that the $\epsilon$-covering number $C(\epsilon)$ of a set $S$ is the smallest
number of Euclidean balls of radius $\epsilon$ required to cover the set. The
$\epsilon$-packing number $P(\epsilon)$ is
the maximum number of sets of the form $B(x,\epsilon) \cap S$, where $x \in S$,
that may be packed into $S$ without overlap.

\begin{lemma}[\textbf{5.2 in \citet{niyogi2008finding}}] 
For every $\epsilon>0$,
$
P(2 \epsilon) \leq C(2 \epsilon) \leq P(\epsilon).
$
\label{NSW2}
\end{lemma}
Combining Lemma \ref{NSW1} and Proposition $\ref{propAlg}$, we obtain an upper bound on $R_n$ for the noiseless case.

\begin{theorem} In the noiseless case ($\sigma=0$), an upper bound on the minimax risk $R_n$ is given by
$$
R_n \leq \frac{8 \ \text{length}(G)}{\delta} \ \exp\left\{-\frac{a\, \delta \, n }{4 \ \text{length}(G)} \right\},
$$
where
\begin{equation}
\delta = \frac{1}{2} \min\left\{\xi \frac{2 \sin(\alpha/4)}{3\sin(\alpha/4)+1}\; , \;  
\frac{\tau \sin(\alpha/2)\sin(\alpha/4)}{ \sin(\alpha/2)\sin(\alpha/4)+ 3 \sin(\alpha/4)+1} \sin\left(\frac{\min\{b,\alpha \tau \}}{2\tau} \right)\right\}.
\label{eq::deltaNoiseless}
\end{equation}
\end{theorem}
\begin{proof}
In the noiseless case, Proposition \ref{propAlg} implies that the graph $G$ can be reconstructed from a $\delta/2$-dense sample $\mathbb{Y}$ if 
\begin{equation}
\delta < \min\left\{\xi \frac{2 \sin(\alpha/4)}{3\sin(\alpha/4)+1}\; , \;  f(b,\alpha,\tau,\xi,0) \right\}.
\label{cond3}
\end{equation}
The value of $\delta$ selected in \eqref{eq::deltaNoiseless} satisfies condition \eqref{cond3}, which follows from conditions \eqref{cond1} and \eqref{cond2}, with $\sigma=0$. We look for the sample size $n$ that guarantees a $\delta/2$-dense sample with high probability. Following the strategy in \citet{niyogi2008finding}, we consider a cover of the metric graph $G$ by balls of radius $\delta/4$. Let $\{x_i: 1\leq i \leq l \}$ be the centers of such balls that constitute a minimal cover. We can choose $A_i^{\delta/4}= B_{\delta/4}(x_i) \cap G$. Applying Lemma \ref{NSW1} we find that the sample size that guarantees a correct reconstruction with probability at least $1-\lambda$ is
\begin{equation}
\label{eq::size1}
\frac{1}{\gamma} \left(\log l + \log \frac{1}{\lambda} \right),
\end{equation}
where $$\displaystyle \gamma \geq \min_i \frac{a \text{ length}(A_i^{\delta/4})}{\text{length}(G)}\geq \frac{a \delta}{4 \text{ length}(G)} \quad, $$ 
and we bound the covering number $l$ in terms of the packing number, using Lemma \ref{NSW2}:
$$\quad \displaystyle l \leq \frac{\text{length}(G)}{\min_i \text{length}(A_i^{\delta/8})} \leq \frac{8 \text{ length}(G)}{\delta}.$$
Therefore, from \eqref{eq::size1}, if
\begin{equation}
n = \frac{4 \text{ length } (G)}{a \delta} \left[\log\left(\frac{8 \text{ length}(G)}{\delta} \right) + \log \frac{1}{\lambda} \right]
\label{size}
\end{equation}
we have a $\delta/2$-dense sample with probability at least $1-\lambda$ and, by Proposition \ref{propAlg}, \\$\mathbb{P}(\hat G  \not\simeq G) \leq \lambda$. Rearranging we have the result.
\end{proof}

Note that, in the noiseless case, the upper and lower bounds
are tight up to polynomial factors in the parameters
$\tau, b, \xi$.
There is a small gap with respect to $\alpha$;
closing this gap is an open problem.\\
In the tubular noise case, we assume that $\sigma$ is small enough, to guarantee that Algorithm \nolinebreak\ref{alg::MGR}  correctly reconstructs a metric graph starting from a $\delta/2$-dense sample.

\begin{theorem} Assume that $\sigma$ satisfies condition \eqref{cond:tube} and $0<\sigma < \min\{3\tau/16, \delta/8\}$, where
\begin{equation}
\label{eq:deltaC0}
\delta= C_0 \min\left\{ \frac{\xi-3\sigma- \tau \sin(\alpha/2) - (\tau-\sigma) \sin(\alpha'/2) }{3/2 +[2 \sin(\alpha'/4)]^{-1}}, \, f(b,\alpha,\tau,\xi,\sigma) \right\},
\end{equation}
for some $0<C_0<1$.
 Under the tubular noise model, an upper bound on the minimax risk $R_n$ is given by
$$
R_n \leq  \frac{16 \text{length}(G)}{\delta}  \exp\left( -\frac{C_D' \delta (\tau-8\sigma) n}{\tau \, \text{length}(G)} \right),
$$
where $C_D'$ is a constant depending on the ambient dimension.
\end{theorem}
\begin{proof}
Proposition \ref{propAlg} implies that the graph $G$ can be reconstructed from a $\delta/2$-dense sample $\mathbb{Y}$ if
\begin{equation}
\delta < \min\left\{ \frac{\xi-3\sigma- \tau \sin(\alpha/2) - (\tau-\sigma) \sin(\alpha'/2) }{3/2 +[2 \sin(\alpha'/4)]^{-1}}, \,  f(b,\alpha,\tau,\xi,\sigma)  \right\},
\label{cond3tube}
\end{equation}
which is satisfied by the value of $\delta$ selected in \eqref{eq:deltaC0}.
We look for the sample size $n$ that guarantees a $\delta/2$-dense sample in $G_\sigma$ with high probability.\\
We consider a cover of the metric graph $G$ by euclidean balls of radius $\delta/8$. Let $\{x_i: 1\leq i \leq l \}$ be the centers of such balls that constitute a minimal cover. 
Note that $D$-dimensional balls of radius $\delta/8+\sigma \leq \delta/4$ centered at the same $x_i's$ constitute a cover of the tubular region $G_\sigma$.
We define $A_i^{\delta/8+\sigma}= B_{\delta/8+\sigma}(x_i) \cap G_\sigma$. Applying Lemma \ref{NSW1} we find that the sample size that guarantees a $\delta/2$-dense sample in $G_\sigma$ (and a correct topological reconstruction of $G$) with probability at least $1-\lambda$ is
\begin{equation}
\label{eq::size2}
\frac{1}{\gamma} \left(\log l + \log \frac{1}{\lambda} \right),
\end{equation}
where 
\begin{equation}
\label{eq::gamma}
\gamma= \min_i \frac{\text{vol}(A_i^{\delta/8+\sigma})}{\text{vol}(G_\sigma)}.
\end{equation}
Define $\tilde{A}_i^{\delta}= B_{\delta}(x_i) \cap G.$ The covering number $l$ is bounded in terms of the packing number, using Lemma \ref{NSW2},
$$
\quad \displaystyle l \leq \frac{\text{length}(G)}{\min_i \text{length}(\tilde{A}_i^{\delta/16})} \leq \frac{16 \text{ length}(G)}{\delta}.
$$
We construct a lower bound on $\gamma$ by deriving an upper bound on the denominator of \eqref{eq::gamma} and a lower bound on the numerator.

\textbf{Upper bound on vol$(G_\sigma)$}. Let $N_\sigma$ be the $\sigma$-covering number of $G$ and let $\mathcal{C}_\sigma$ be the set of centers of this cover. By Lemma \ref{NSW2}, $N_\sigma$ is bounded by the $\sigma/2$-packing number. A simple volume argument gives
$
N_\sigma \leq C \text{length}(G)/\sigma,
$
for some constant $C$. Note that $2\sigma$ $D$-dimensional balls around each of the centers in $\mathcal{C}_\sigma$ cover $G_\sigma$. Thus vol$(G_\sigma) \leq v_DN_\sigma (2\sigma)^D \leq C_D \text{length}(G) \sigma^{D-1}$ for some constant $C_D$ depending on the ambient dimension.

\textbf{Lower bound on vol$(A_i^{\delta/8+\sigma})$, for all $i$}. 
Let $P_{A}(\sigma)$ be the $\sigma$-packing number of $\tilde{A}_i^{\delta/8}$ and let $\mathcal{D}_{A}$ be the set of centers of this packing. Then vol$(A_i^{\delta/8+\sigma})\geq P_{A}(\sigma) v_D \sigma^D$, because the union of $\sigma$ balls around $\mathcal{D}_{A}$ is contained in $A_i^{\delta/8+\sigma}$.
Let $C_{A}(2\sigma)$ be the $2\sigma$-covering number of $\tilde{A}_i^{\delta/8}$ and let $\mathcal{C}_{A}=\{z_1,\dots, z_{C_A(2\sigma)} \}$ be the set of centers of this cover. By Lemma \ref{NSW2},
$$
P_A(\sigma) \geq C_A(2\sigma) \geq \frac{\text{length}(\tilde{A}_i^{\delta/8})}{ \max_{z_j \in \mathcal{C}_A} \text{length}(B_{2\sigma}(z_j) \cap \tilde{A}_i^{\delta/8}) } \geq \frac{\delta/8}{ \max_{z_j \in \mathcal{C}_A} \text{length}(B_{2\sigma}(z_j) \cap \tilde{A}_i^{\delta/8}) } 
$$
and, since $2\sigma<3\tau/8$, by Corollary 1.3 in \cite{fred2013Note},
$$
\max_{z_j \in \mathcal{C}_A} \text{length}(B_{2\sigma}(z_j) \cap \tilde{A}_i^{\delta/8})  \leq C_2 \left( \frac{\tau}{\tau-8\sigma} \right) \sigma,
$$
for some constant $C_2$.
Thus
$$
\gamma \geq  \frac{P_A(\sigma) v_D \sigma^D}{C_D \text{length}(G) \sigma^{D-1}} \geq   C_D'  \frac{\delta (\tau-8\sigma)}{\tau \text{length}(G) }
\quad,
$$ 
where $C_D'$ is a constant depending on the ambient dimension.\\

Finally, from \eqref{eq::size2}, if
\begin{equation}
n = \frac{\tau \, \text{ length } (G)}{ C_D' \delta (\tau-8\sigma)} \left[\log\left(\frac{16 \text{ length}(G)}{\delta} \right) + \log \frac{1}{\lambda} \right],
\label{size}
\end{equation}
then the sample is $\delta/2$-dense with probability at least $1-\lambda$ and $\mathbb{P}(\hat G  \not\simeq G) \leq \lambda$. Rearranging we obtain
$$
R_n \leq \exp\left( -\frac{C_D' \delta (\tau-8\sigma) n}{\tau \, \text{length}(G)} + \log\left(\frac{16 \text{length}(G)}{\delta } \right) \right).
$$
\end{proof}

\section{Discussion}
\label{section::conclusion}

In this paper, we presented a statistical analysis of metric graph reconstruction.
We derived sufficient conditions on random samples from a graph metric space that guarantee topological reconstruction
and we derived lower and upper bounds on the minimax
risk for this problem. Various improvements and theoretical extensions are possible.
In Proposition \ref{propAlg} we have analyzed Algorithm \ref{alg::MGR} using the Euclidean distance at every step. It is possible to obtain a similar result using a different notion of distance, for example, the distance induced by a Rips-Vietoris graph constructed on the sample.

While in our analysis we mainly relied on the
assumption of a dense sample, \citet{aanjaneya2012metric} used the
more refined but stronger assumption of the sample being an
approximation of the metric graph, which we recall:
given positive numbers $\varepsilon$ and $R$, we say that
$(\mathbb{Y},d_{\mathbb{Y}})$ is an \textit{$(\varepsilon,R)$-approximation}
of the metric space $(G, d_G)$ if there exists a correspondence 
$C\subset G\times \mathbb{Y}$ such that
\begin{equation}
(x,y),(x',y')  \in C, \min(d_{G}(x,x'), d_{\mathbb{Y}}(y,y')) \leq R \;\; \Longrightarrow 
\left| d_{G}(x,x')- d_{\mathbb{Y}}(y,y') \right| \leq \varepsilon.
\end{equation}
As shown in \citet{aanjaneya2012metric}, the
$(\varepsilon,R)$-approximation assumption is sufficient, for appropriate
choice of the parameters $\varepsilon$ and $R$, to recover not only the
topology of a metric graph $(G,d_G)$, but also its metric $d_G$ with
high accuracy. However, when compared to the dense sample assumption,
it demands a larger sample complexity to achieve accurate
topological reconstruction. A strategy similar to the one used in this paper could be used to determine
the sample size that guarantees an $(\varepsilon, R)$-approximation of the underlying metric graph with high probability.
This would guarantee a correct topological reconstruction, as well as an approximation of the metric $d_G$.

We are also investigating
the idea of combining metric graph reconstruction with
the subspace constrained mean-shift algorithm 
\citep{ fukunaga1975estimation, comaniciu2002mean, genovese2012nonparametric} to provide similar guarantees.
Our preliminary results indicate that this mixed strategy works very well under more general noise assumptions and with relatively low sample size.


\acks{Research supported by NSF CAREER Grant DMS 114967, Air Force Grant FA95500910373, NSF Grant DMS-0806009. The authors thank the referees for helpful comments and suggestions.}

\vskip 0.2in
\bibliography{paper}
\end{document}